\documentclass[12pt]{article} % for a short document

\usepackage[margin=3cm]{geometry}
\usepackage[english]{babel}
\usepackage{hyperref}
\usepackage{amsmath}
\usepackage{amsthm}
\usepackage{amssymb}
\usepackage{enumerate}
\usepackage[utf8]{inputenc}
\usepackage[T1]{fontenc}
 \usepackage[affil-it]{authblk}
 
\usepackage{tikz}
\usetikzlibrary{arrows}
\usetikzlibrary{decorations.markings}
\pgfdeclarelayer{nodelayer}
\pgfdeclarelayer{edgelayer}
\pgfsetlayers{nodelayer,main,edgelayer}
\tikzstyle{none}=[draw=none]   
\tikzstyle{bigtiparrow}=[->,thick, >=angle 90]
\tikzstyle{bigtiparrow2}=[->,thick, >=angle 90,preaction={draw=white, -,line width=6pt}]
\tikzstyle{lrarrow}=[<->,thick, >=angle 90,preaction={draw=white, -,line width=6pt}]

\tikzstyle{new}=[rectangle,fill=white,draw=white, inner sep=2pt]
\tikzstyle{new2}=[rectangle,fill=white,draw=white, inner sep=6pt]

\DeclareMathOperator{\rng}{\mathrm{rng}}
\DeclareMathOperator{\supp}{\mathrm{supp}}
\DeclareMathOperator{\Cost}{\mathrm{Cost}}

\newcommand{\Aut}{\mathrm{Aut}}

  \newcommand{\N}{\mathbb N}
  
  \newcommand{\Z}{\mathbb Z}

 \newcommand{\dom}{\mathrm{dom}\;}

  \newcommand{\inv}{^{-1}}

  \renewcommand{\leq}{\leqslant}
  \renewcommand{\geq}{\geqslant}

  \newcommand{\act}{\curvearrowright}

  \newcommand{\la}{\left\langle}
  \newcommand{\ra}{\right\rangle}

\newtheorem{thm}{Theorem}
\newtheorem{cor}[thm]{Corollary}

\newtheorem{prop}[thm]{Proposition}

\theoremstyle{definition}
\newtheorem*{claim}{Claim}

\newtheorem{df}[thm]{Definition}
\newtheorem*{rmq}{Remark}

\title{The number of topological generators for full groups of ergodic equivalence relations}
\author{François Le Maître\footnote{Research partially supported by ANR grant AGORA (ANR-09-BLAN-0059)}}
\date{}
\begin{document}

\maketitle

\begin{abstract}
We completely elucidate the relationship between two invariants associated with an ergodic probability measure-preserving (pmp) equivalence relation, namely its cost and the minimal number of topological generators of its full group. It follows that for any free pmp ergodic action of the free group on $n$ generators, the minimal number of topological generators for the full group of the action is $n+1$, answering a question of Kechris.
\end{abstract}

Let $\Gamma$ be a countable group of measure-preserving automorphisms on a standard probability space $(X,\mu)$, in other words a countable subgroup of $\Aut(X,\mu)$. Classical ergodic theory is concerned with the conjugacy class of $\Gamma$, especially when $\Gamma=\Z$. A natural invariant is the partition of $X$ into orbits induced by $\Gamma$, that is,  the orbit equivalence relation $x \,\mathcal R_\Gamma\,y$ iff there exists an element $\gamma\in\Gamma$ such that $x=\gamma y$. This probability measure-preserving (pmp) equivalence relation is entirely captured by its full group $[\mathcal R_\Gamma]$, defined to be the group of measure-preserving automorphisms $\varphi$ such that $\varphi(x) \,\mathcal R_\Gamma\, x$ for all $x\in X$. So while classical ergodic theory studies the conjugacy class of $\Gamma$ in $\Aut (X,\mu)$, orbit equivalence theory considers the conjugacy class  of the much bigger group $[\mathcal R_\Gamma]$ in $\Aut(X,\mu)$.

This motivates the study of $[\mathcal R_\Gamma]$ as a topological group in its own right. Let $d_u$ be the bi-invariant complete metric on $\Aut(X,\mu)$ defined by $$d_u(T,U)=\mu(\{x\in X: T(x)\neq U(x)\}).$$
Then $[\mathcal R_\Gamma]$ is closed and separable for this metric, in other words it is a Polish group. Some of its topological properties are relevant to the study of $\mathcal R_\Gamma$. For instance, a theorem of Dye \cite[prop. \!5.1]{MR0158048} asserts that closed normal subgroups of $[\mathcal R_\Gamma]$ are in a natural bijection with $\Gamma$-invariant subsets of $X$. In particular, $[\mathcal R_\Gamma]$ is topologically simple iff $\mathcal R_\Gamma$ is ergodic. Another example is given by amenability:  Giordano and Pestov showed that $[\mathcal R_\Gamma]$  is  extremely amenable iff $\mathcal R_\Gamma$ is amenable \cite[thm. \!5.7]{MR2311665}.

If $G$ is a separable topological group, a natural invariant of $G$ is its minimal number of topological generators $t(G)$, that is, the least $n\in\N\cup\{\infty\}$ such that there are $g_1,...,g_n\in G$ which generate a dense subgroup of $G$. For full groups of pmp equivalence relations, the investigations on this number were started by a question of Kechris  \cite[4.(D)]{MR2583950}. Kittrell and Tsankov provided a partial answer to this question, using the theory of cost initiated by Levitt \cite{MR1366313} and developped by Gaboriau in the seminal article \cite{MR1728876}. They showed that if $\mathcal R$ is a pmp ergodic equivalence relation, then the cost of $\mathcal R$ is finite iff $[\mathcal R]$ is topologically finitely generated \cite[thm. \!1.7]{MR2599891}. Using a theorem of Matui \cite[ex. \!6.2]{MR2205435}, they also showed that for an ergodic $\Z$-action, the associated orbit equivalence relation $\mathcal R_\Z$ satisfies $2\leq t([\mathcal R_\Z])\leq 3$,
and deduced the explicit bounds $$\lfloor\Cost (\mathcal R)\rfloor+1\leq t([\mathcal R])\leq 3(\lfloor\Cost(\mathcal R)\rfloor+1)$$
for an arbitrary pmp ergodic equivalence relation $\mathcal R$. These bounds were later refined by Matui \cite[thm. \!3.2]{matuifullgp2} who showed that the number of topological generators for the full group of an ergodic equivalence relation generated by a $\Z$-action is in fact equal to two, which implies that
$$\lfloor\Cost (\mathcal R)\rfloor+1\leq t([\mathcal R])\leq 2(\lfloor\Cost(\mathcal R)\rfloor+1).$$

In this paper, we show that the minimal number of topological generators for the full group of $\mathcal R$ is actually completely determined by the cost of $\mathcal R$.

\begin{thm}\label{thethm}
Let $\mathcal R$ be a pmp ergodic equivalence relation. Then the minimal number $t([\mathcal R])$ of topological generators for the full group of $\mathcal R$ is related to the integer part of the cost of $\mathcal R$ by the formula
$$t([\mathcal R])=\lfloor\Cost(\mathcal R)\rfloor+1.$$
\end{thm}

Using a theorem of Gaboriau (cf. theorem \ref{oefn}), we get the following corollary:

\begin{cor}Let $\mathbb F_n\act(X,\mu)$ be a free pmp ergodic action. Then the minimal number of topological generators for the full group of this action is $n+1$.
\end{cor}

Theorem \ref{thethm} also yields the first examples of nonamenable equivalence relations whose full group is topologically 2-generated. For instance, one can use theorem VI.26.(b) of \cite{MR1728876} and see that:

\begin{cor}If $n\geq 3$, then every free pmp ergodic action of $Sl_n(\Z)$ yields a full group whose minimal number of topological generators is $2$.
\end{cor}

Moreover,  if $\mathcal R$ is any ergodic pmp equivalence relation, one can fully recover the cost of $\mathcal R$ as the infimum over topological generators of $[\mathcal R]$ of the sums of the measures of their supports (cf. theorem \ref{coutgentop}). \\

Note that in 2011 Marks showed that pmp modular actions of $\mathbb F_n$ have a full group with $n+1$ topological generators (cf. corrections and updates of \cite{MR2583950}). The question of the number of topological generators for non ergodic aperiodic equivalence relations will be treated in a forthcoming paper, where it will be proved for instance that for free pmp actions of $\mathbb F_n$, one still has $t([\mathcal R_{\mathbb F_n}])=n+1$. 

The paper is organised as follows: section \ref{defnot} contains standard definitions and notations used throughout the paper as well as a rough review of the orbit equivalence theory we need, section \ref{rphi} introduces the notion of pre-$p$-cycle, which is then used to prove theorem \ref{thethm} in section \ref{theproof}. 

\section{Definitions and notations}\label{defnot}

Everything will be understood ‘‘modulo sets of measure zero''. We now review some standard notation and definitions.

If $(X,\mu)$ is a standard probability space, and $A,B$ are Borel subsets of $X$, a \textbf{partial isomorphism} of $(X,\mu)$ of \textbf{domain} $A$ and \textbf{range} $B$ is a Borel bijection $f: A\to B$ which is measure-preserving for the measures induced by $\mu$ on $A$ and $B$ respectively. We denote by $\dom f=A$ its domain, and by $\rng f=B$ its range. Note that in particular, $\mu(\dom f)=\mu(\rng f)$. A \textbf{graphing} is a countable set of partial isomorphisms of $(X,\mu)$, denoted by $\Phi=\{\varphi_1,...,\varphi_k,...\}$ where the $\varphi_k$'s are partial isomorphisms. It \textbf{generates} a \textbf{measure-preserving equivalence relation} $\mathcal R_\Phi$, defined to be the smallest equivalence relation containing $(x,\varphi(x))$ for every $\varphi\in \Phi$ and $x\in\dom\varphi$.  The \textbf{cost} of a graphing $\Phi$ is the sum of the measures of the domains of the partial isomorphisms it contains. The \textbf{cost} of a measure-preserving equivalence relation $\mathcal R$ is the infimum of the costs of the graphings that generate it, we denote it by $\Cost(\mathcal R)$. 
The \textbf{full group} of $\mathcal R$ is the group $[\mathcal R]$ of automorphisms of $(X,\mu)$ which induce permutations in the $\mathcal R$-classes, that is
$$[\mathcal R]=\{\varphi\in\Aut(X,\mu): \forall x\in X, \varphi(x)\,\mathcal R\, x\}.$$
It is a separable group when equipped with the complete metric $d_u$ defined by $$d_u(T,U)=\mu(\{x\in X: T(x)\neq U(x)\}.$$ One also defines the \textbf{pseudo full group} of $\mathcal R$, denoted by $[[\mathcal R]]$, which consists of all partial isomorphisms $\varphi$ such that $\varphi(x)\, \mathcal R \, x$ for all $x\in \dom\varphi$. 

Say that a measure-preserving equivalence relation $\mathcal R$ is \textbf{ergodic} when every Borel $\mathcal R$-saturated set has measure $0$ or $1$. We will use without mention the following standard fact about ergodic measure-preserving equivalence relations:

\begin{prop}[see e.g. \cite{MR2095154}, lemma 7.10.]\label{isopar}
Let $\mathcal R$ be an ergodic measure-preserving equivalence relation on $(X,\mu)$, let $A$ and $B$ be two Borel subsets of $X$ such that $\mu(A)=\mu(B)$. Then there exists $\varphi\in[[\mathcal R]]$ of domain $A$ and range $B$.
\end{prop}

We now give a very sparse outline of orbit equivalence theory (for a survey, see \cite{MR2827853}). For ergodic $\Z$-actions, it turns out that there is only one orbit equivalence relation: a theorem of Dye \cite[thm. \!5]{MR0131516} states that if $(X,\mu)$ is a standard probability space and $T,U$ are ergodic automorphisms of $(X,\mu)$ then their full groups are conjugated by an automorphism of $(X,\mu)$. This was later generalized by Ornstein and Weiss \cite{MR551753} who  proved that all amenable groups induce the same ergodic orbit equivalence relation.
 
Actually, inducing only one orbit equivalence relation characterizes amenable groups: if $\Gamma$ is a non amenable group, then it has continuum many non orbit equivalent actions, as was shown by Epstein \cite{MR2711968}, building on earlier work by Gaboriau, Lyons and Ioana (\cite{MR2534099} and  \cite{MR2810796}).

Last but not least, Gaboriau has computed the cost for free pmp actions of free groups, and deduced that two free pmp actions of $\mathbb F_n$ and $\mathbb F_m$ cannot be orbit equivalent if $n\neq m$.
\begin{thm}[\cite{MR1728876}, corollaire 1]\label{oefn} If $\mathbb F_n\act(X,\mu)$ is a free measure-preserving action, then the associated orbit equivalence relation has cost $n$.
\end{thm}

The proof of our theorem will rely on a result of Matui which gives two topological generators for the full group of some ergodic $\Z$-action.  His construction actually yields the following stronger statement.

\begin{thm}[\cite{matuifullgp2}, theorem 3.2]\label{EO}
For every $\epsilon>0$,  there exists an ergodic automorphism $T$ of $(X,\mu)$ and an involution $U\in [ \mathcal R_T]$ with support of measure less than $\epsilon$  such that $$\overline{\la T,U\ra}=[\mathcal R_T].$$
\end{thm}

We also borrow the following result from Kittrell and Tsankov.

\begin{thm}[\cite{MR2599891}, theorem 4.7]\label{ktdense}
Let $\mathcal R_1$, $\mathcal R_2$,... be measure-preserving equivalence relations on $(X,\mu)$, and let $\mathcal R$ be their join (i.e. the smallest equivalence relation containing all of them). Then $\la\bigcup_{n\in\N}[\mathcal R_n]\ra$ is dense in $[\mathcal R]$.
\end{thm}

\section{Pre-$p$-cycles, $p$-cycles and associated full groups}\label{rphi}

Let us give a way to construct pmp equivalence relations whose nontrivial classes have cardinality $p$.

\begin{df} Let $p\in\N$. A \textbf{pre}-$p$-\textbf{cycle} is a graphing $\Phi=\{\varphi_1,...,\varphi_{p-1}\}$ such that  the following two conditions  are satisfied:
\begin{enumerate}[(i)]\item  $\forall i\in\{1,...,p-1\}, \rng\varphi_i=\dom\varphi_{i+1}$.
\item The following sets are all disjoint:
 $$\dom\varphi_1, \dom \varphi_2,...,\dom\varphi_{p-1},\rng\varphi_{p-1}.$$
\end{enumerate}
\end{df}

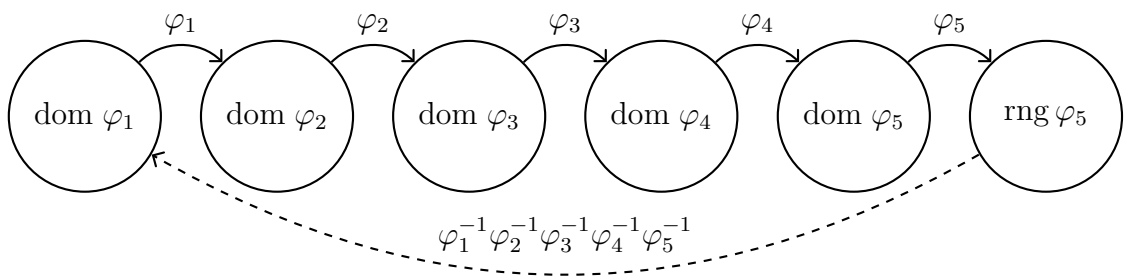
\begin{figure}[htbp]
\centering
\tikzstyle{bigtiparrow}=[->,thick, >=angle 90]
\tikzstyle{input}=[circle,
                                    thick,
                                    minimum size=2cm,
                                    inner sep=0pt,
                                    draw=black,
                                    fill=white]
\tikzstyle{points}=[circle,
                                    thick,
                                    minimum size=2cm,
                                    draw=white,
                                    fill=white]                                    
\tikzstyle{edgy}=[thick, bend left=45, >=angle 90]

\begin{tikzpicture}%[>=latex,text height=1.5ex,text depth=0.25ex]
    % "text height" and "text depth" are required to vertically
    % align the labels with and without indices.
  
  % The various elements are conveniently placed using a matrix:
  \matrix[row sep=1cm,column sep=0.5cm] {
        \node (A1) [input]{$\dom \varphi_1$}; 
        &
        \node (A2)   [input]{$\dom\varphi_2$};     
        &
        \node (A3)   [input]{$\dom\varphi_3$};     

        &        
    \node(A4)[input]{$\dom\varphi_4$};
    & \node (Ap1) [input]{$\dom\varphi_{5}$};
        &
        \node (Ap) [input]{$\rng\varphi_{5}$};

        \\
     };

\draw [bigtiparrow] (A1) to [bend  left=45] node[above] {$\varphi_1$} (A2);

\draw [bigtiparrow] (A2) to [bend  left=45] node[above] {$\varphi_2$} (A3);

\draw [bigtiparrow] (A3) to [bend  left=45] node[above] {$\varphi_3$} (A4);
\draw [bigtiparrow] (A4) to [bend  left=45] node[above] {$\varphi_{4}$} (Ap1);
\draw [bigtiparrow] (Ap1) to [bend  left=45] node[above] {$\varphi_{5}$} (Ap);
\draw [bigtiparrow, dashed] (Ap) to [bend  left=30] node[ above=0.15cm] {$\varphi_1\inv\varphi_2\inv\varphi_3\inv\varphi_4\inv\varphi_{5}\inv$} (A1);

\end{tikzpicture}

\caption{A pre-$6$-cycle $\Phi=\{\varphi_1,...,\varphi_{5}\}$. The dashed arrow shows how to get the corresponding $6$-cycle $C_\Phi$ (cf. the paragraph following definition \ref{cycle}).}
\end{figure}

\begin{rmq}
Given a graphing $\Phi$, the fact that it is a pre-$p$-cycle is witnessed by a unique enumeration of its elements. When we have  a pre-$p$-cycle $\Phi$ and write it as $\Phi=\{\varphi_1,...,\varphi_{p-1}\}$, we will always assume that conditions (i) and (ii) above are satisfied.
\end{rmq}

\begin{df}\label{cycle} A $p$\textbf{-cycle} is an element $C\in \Aut(X,\mu)$ whose orbits have cardinality $1$ or $p$.
\end{df}

Given a pre-$p$-cycle $\Phi=\{\varphi_1,...,\varphi_{p-1}\}$, we define a $p$-cycle $C_\Phi\in\Aut(X,\mu)$ as follows:
$$C_\Phi(x)=\left\{\begin{array}{ll}\varphi_i(x) & \text{if }x\in \dom\varphi_i\text{ for some }i<p, \\
\varphi_1\inv\varphi_2\inv\cdots\varphi_{p-1}\inv(x) &\text{if }x\in \rng\varphi_{p-1},\\
x & \text{otherwise.}\end{array}\right.$$

Note that $\{C_\Phi\}$ generates the pmp equivalence relation $\mathcal R_\Phi$, whose classes have cardinality either $1$ or $p$. We now give a useful generating set for the full group of $\mathcal R_\Phi$. 

\begin{prop}\label{isopgen}
If $\Phi=\{\varphi_1,...,\varphi_{p-1}\}$ is a pre-$p$-cycle, then for all $i\in\{1,...,p-1\}$, the full group of $\mathcal R_\Phi$ is topologically generated by $[\mathcal R_{\{\varphi_i\}}]\cup\{C_\Phi\}$.
\end{prop}
\begin{proof} By definition, the equivalence relation $\mathcal R_\Phi$ is the join of the equivalence relations $\mathcal R_{\{\varphi_1\}}, ...,\mathcal R_{\{\varphi_{p-1}\}}$. Then by theorem \ref{ktdense}  we only need to show that for all $j\in\{1,..,p-1\}$, $[\mathcal R_{\{\varphi_j\}}]$ is contained in the group generated by $[\mathcal R_{\{\varphi_i\}}]\cup\{C_\Phi\}$. But for all $j\in\{1,...,p-2\}$, we have the conjugation relation
$$C_\Phi\varphi_jC_\Phi\inv= \varphi_{j+1},$$
which in turn yields $C_\Phi[\mathcal R_{\{\varphi_j\}}]C_{\Phi}\inv=[\mathcal R_{\{\varphi_{j+1}\}}]$. So the group generated by $[\mathcal R_{\{\varphi_i\}}]\cup\{C_\Phi\}$ contains all $[\mathcal R_{\{\varphi_j\}}]$ for $1\leq j\leq p-1$.
\end{proof}

\section{Proof of the theorem}\label{theproof}

We now get to the proof of our \hyperref[thethm]{main result}. Fix an ergodic measure-preserving equivalence relation $\mathcal R$ on $(X,\mu)$.\\

Let us first show that $t([\mathcal R])\geq \lfloor\Cost(\mathcal R)\rfloor+1$, following an argument of Miller \cite[cor. 4.12]{MR2599891}. This inequality is equivalent to $t([\mathcal R])>\Cost(\mathcal R)$.  Observe that topological generators of $[\mathcal R]$ must also generate the equivalence relation $\mathcal R$, so the definition of cost yields $t([\mathcal R])\geq\Cost(\mathcal R)$. Assume, towards a contradiction, that  $t([\mathcal R])=\Cost(\mathcal R)$ and fix $\Cost(\mathcal R)=n$ topological generators of $[\mathcal R]$. Then again, these elements also generate $\mathcal R$ and thus realize the cost of $\mathcal R$. By proposition I.11 in \cite{MR1728876}, they  induce a free action of $\mathbb F_n$ on $(X,\mu)$.
 In particular, the countable group they generate is discrete in $[\mathcal R]$, hence closed, which contradicts its density.\\

The other inequality requires more work. Suppose that for a fixed $n\in\N$, $\Cost \mathcal R<n+1$. We want to find $n+1$ topological generators of $[\mathcal R]$. It is a standard fact that we can find a pmp ergodic equivalence relation $\mathcal R_0\subseteq \mathcal R$  generated by a single automorphism (cf. for instance  \cite[prop. \!\!3.5]{MR2583950}).

Let $\Phi_0$ be a graphing of cost 1 which generates $\mathcal R_0$. Lemma III.5 in \cite{MR1728876} provides a graphing $\Phi$ such that $\Cost(\Phi)<n$ and $\Phi_0\cup\Phi$ generates $\mathcal R$. Let 
$$c=\frac{\Cost(\Phi)}{n}<1,$$
and fix some odd $p\in\N$ such that $(\frac {p+2} p)c<1$. Splitting the domains of the partial automorphisms in $\Phi$, we find $\Phi_1$,...,$\Phi_n$ of cost $c$ such that $\Phi=\Phi_1\cup\cdots\cup \Phi_n$.

By theorem \ref{EO} for $\epsilon=1-(\frac {p+2}p)c$, there is an ergodic automorphism $T_0$ of $(X,\mu)$ and an involution $U_0\in[\mathcal R_{T_0}]$, whose support has measure less than $\epsilon$, such that $T_0$ and $U_0$ topologically generate $[\mathcal R_{T_0}]$. Dye's theorem \cite[thm. \!\!5]{MR0131516} yields $\varphi\in\Aut(X,\mu)$ which conjugates $[\mathcal R_{T_0}]$ and $[\mathcal R_0]$, so up to conjugation by $\varphi$ we can assume that $\mathcal R_{T_0}=\mathcal R_0$.

Let $A_1,...,A_{p+2}$ be disjoint subsets of $X\setminus\supp U_0$, of measure $\frac c p$ each. Recall proposition \ref{isopar}, and note that up to pre- and post-composition of the partial isomorphisms of each $\Phi_i$ by elements in $[[\mathcal R_0]]$, we can assume that each $\Phi_i$ is a pre-$(p+1)$-cycle $\Phi_i=\{\varphi_1^i,\varphi_2^i,...,\varphi_p^i\}$ such that $\varphi^i_j: A_j\to A_{j+1}$.

Now choose $\psi\in[[\mathcal R_0]]$ with domain $A_{p+1}$ and range $A_{p+2}$, and add it to every $\Phi_i$. We get $n$ pre-$(p+2)$-cycles $\tilde\Phi_i=\Phi_i\cup\{\psi\}$, and $\Phi_0\cup\tilde\Phi_1\cup\cdots\cup \tilde\Phi_n$ still generates $\mathcal R$. Consider the associated $(p+2)$-cycles $C_{\tilde\Phi_i}$.

\begin{claim}
The $(n+2)$ elements $T_0,U_0,C_{\tilde\Phi_1},...,C_{\tilde\Phi_n}$ topologically generate the full group of $\mathcal R$.
\end{claim}
\begin{proof}
Let $G$ be the closed group generated by $T_0,U_0,C_{\tilde\Phi_1},...,C_{\tilde\Phi_n}$. Recall that $T_0$ and $U_0$ have been chosen so that they topologically generate the full group of $\mathcal R_0$, so $G$ contains $[\mathcal R_0]$. 
Because $\psi$ is a partial isomorphism of $\mathcal R_0$, we get $[\mathcal R_{\{\psi\}}]\subseteq[\mathcal R_0]\subseteq G$. Then, since for all $i\in\{1,...,n\}$ we have $\psi\in\tilde\Phi_i$ and $C_{\tilde\Phi_i}\in G$, proposition \ref{isopgen} implies that $G$ contains $[\mathcal R_{\tilde\Phi_i}]$.
But $\mathcal R$ is the join of $\mathcal R_0$, $\mathcal R_{\tilde\Phi_1},...,\mathcal R_{\tilde\Phi_n}$, so by theorem \ref{ktdense} we are done.
\end{proof}

Then observe that the involution $U_0$ and the $(p+2)$-cycle $C_{\tilde \Phi_1}$ have disjoint support, so they commute and the product $U_1=U_0C_{\tilde\Phi_1}$ satisfies
$$(U_0C_{\tilde\Phi_1})^n=U_0^nC_{\tilde\Phi_1}^n\;\;\;\;\;\forall n\in\N.$$
In particular,  
\begin{align*}
U_1^{p+2}&=(U_0C_{\tilde\Phi_1})^{p+2}=U_0^{p+2}=U_0\;\;\;\;\text{ since }p\text{ is odd and}\\
U_1^{p+3}&=(U_0C_{\tilde\Phi_1})^{p+3}=U_0^{p+3}C_{\tilde\Phi_1}^{p+3}=C_{\tilde\Phi_1},
\end{align*}
so that in fact $T_0,U_1,C_{\tilde\Phi_2},...,C_{\tilde\Phi_n}$ topologically generate $[\mathcal R]$.\\

Note that the construction yields the following statement, which says that the cost is an invariant of the metric group $([\mathcal R],d_u)$.

\begin{thm}\label{coutgentop}Let $\mathcal R$ be a pmp ergodic equivalence relation. Then we have the formula
$$\Cost (\mathcal R)= \inf\left\{\sum_{i=1}^{t([\mathcal R])}d_u(T_i,\mathrm{id}): \overline{\la T_1,...,T_{t([\mathcal R])}\ra}=[\mathcal R]\right\}.$$
\end{thm}

\noindent\textbf{Acknowledgments.} 
I am very grateful to Damien Gaboriau and Julien Melleray for their constant help and support. Theorem \ref{thethm} would certainly have read $$t([\mathcal R])\leq \lfloor \Cost (\mathcal R)\rfloor+2$$ without the enlightening conversations I have had with Damien, and the proof would have been quite different. Finally, I thank Damien, Julien and Adriane Kaïchouh for proofreading  an undisclosed amount of preliminary versions of this article. 

\bibliographystyle{alpha}
\bibliography{/Users/francoislemaitre/ownCloud/Maths/biblio}

\end{document}